%% file: disguised.tex
\documentclass[10pt]{article}

\usepackage{amsmath,amsfonts,amssymb,amsthm}
\usepackage{enumerate}
\usepackage[hyphens]{url}
\usepackage{tcolorbox}
\usepackage{hyperref}
\usepackage{xspace}
\usepackage[margin=3.5cm]{geometry}
\usepackage{marvosym}

\theoremstyle{plain}
\newtheorem{theorem}{Theorem}
\newtheorem{lemma}[theorem]{Lemma}
\newtheorem{corollary}[theorem]{Corollary}
\newtheorem{proposition}[theorem]{Proposition}

\theoremstyle{definition}
\newtheorem{definition}[theorem]{Definition}

\newtheorem{remark}[theorem]{Remark}

\newcommand\R{\mathbb{R}}

\newcommand{\dd}[2]{\frac{\textnormal{d} #1}{\textnormal{d} #2}}

\DeclareMathOperator{\im}{im} 

\DeclareMathOperator{\e}{e} 
\DeclareMathOperator{\diag}{diag}

\newcommand{\ones}{1}
\newcommand{\trans}{\mathsf{T}}

\newcommand{\de}{\stackrel{\text{DE}}{=}}

\definecolor{darkgreen}{rgb}{0.0,0.7,0.0}

\newcommand\gray[1]{{\textcolor{gray}{#1}}}
\newcommand\newdef[1]{#1}

\newcommand\blfootnote[1]{%
  \begin{NoHyper}
  \renewcommand\thefootnote{}\footnote{#1}%
  \addtocounter{footnote}{-1}%
  \end{NoHyper}
}

\usepackage{algorithm, algpseudocode}

\makeatletter

\begin{document}

\title{Disguised complex balance  
via positive algebraic geometry}

\author{
Stefan M\"uller\textsuperscript{\Letter}, 
Abhishek Deshpande, 
Georg Regensburger
}

\maketitle

\begin{abstract}
We study dynamical systems arising from reaction networks under mass-action kinetics. For certain choices of the rate constants (parameters), such systems are complex-balanced (vertex-balanced), which guarantees the existence of a unique positive equilibrium. Moreover, this equilibrium is asymptotically stable (admitting a global Lyapunov function) and linearly stable.

In a series of recent papers, Craciun and collaborators introduced and studied {\em disguised} complex-balanced systems, that is, mass-action systems that are dynamically equal to auxiliary complex-balanced systems and therefore inherit their strong stability properties. Determining the parameter values for which a given system is disguised complex-balanced is a nontrivial algebraic problem.

In this work, we show that the defining conditions for disguised complex-balanced equilibria naturally give rise to parametrized systems of polynomial {\em inequalities}. Using the framework for positive algebraic geometry developed by Müller and Regensburger, we reformulate these systems as binomial equations (on the {\em disguised complex-balanced flux cone}). 
Computing the {\em disguised complex-balanced parameter locus} can be viewed as a quantifier-elimination problem, and our approach eliminates the concentrations (state variables) from the problem. 

We illustrate our results using the running example of a recent paper by Boros et al.

\vspace{2ex}

\vspace{2ex}
\noindent
{\bf MSC classes.} 
37N25,
34C08,
14P10,
14Q30,
92C42
\end{abstract}

\blfootnote{
\scriptsize

\noindent
{\bf Stefan~M\"uller} \\
Faculty of Mathematics, University of Vienna, Oskar-Morgenstern-Platz 1, 1090 Wien, Austria \\
\Letter \, st.mueller@univie.ac.at \\[1ex]
{\bf Abhishek Deshpande} \\
Center for Computational Natural Sciences and Bioinformatics,
International Institute of Information Technology Hyderabad (IIIT),
Hyderabad, Telangana 500032,
India \\[1ex]
{\bf Georg Regensburger} \\
Research Institute for Symbolic Computation (RISC),
Johannes Kepler University Linz, Kirchenplatz 5, 4232 Hagenberg, Austria
}


\clearpage

\section{Introduction}

\color{black}

Every polynomial dynamical system can be viewed as a reaction network with mass-action kinetics, and reaction network theory provides powerful tools for studying the existence, uniqueness, and stability of equilibria. In particular, if a mass-action system is complex-balanced (or vertex-balanced, in modern terminology~\cite{CraciunMuellerPanteaYu2019}), then every invariant subspace (stoichiometric compatibility class) contains a unique positive equilibrium~\cite{HornJackson1972}, which is moreover asymptotically stable~\cite{HornJackson1972} and linearly stable~\cite{Johnston2011,Feinberg2019,Boros2019}. On the one hand, this provides a powerful criterion guaranteeing stable dynamics. On the other hand, complex balance imposes severe structural restrictions: it implies that the underlying reaction network is weakly reversible~\cite{Horn1972}, that is, every connected component of the underlying graph is strongly connected.


Since the same dynamical system can be generated by different reaction networks (with different rate constants)~\cite{HarsToth1987,CraciunPantea2008,SzederkenyiHangos2011,JohnstonEtAl2012}, a given mass-action system need not itself be complex-balanced, but may admit a dynamically equal realization that is. The systematic study of such {\em disguised}~\cite{BrustengaEtAl2022} complex-balanced systems was initiated by Craciun and collaborators in~\cite{CraciunJinYu2020}. Since then, the disguised complex-balanced parameter locus has been investigated from several perspectives. For example, it has been shown to be path connected~\cite{CraciunEtAl2024}, and its dimension has been determined~\cite{CraciunEtAl2025}. More recently, Boros et al.~\cite{BorosEtAl2026} introduced a flux-based approach, providing new geometric insight into the disguised complex-balanced parameter locus and substantially simplifying computations. For more background and references, we refer to~\cite{CraciunJinYu2020,BorosEtAl2026}. Nevertheless, obtaining an explicit characterization of the disguised complex-balanced parameter locus remains a nontrivial quantifier-elimination problem.


In this paper,
we first formulate disguised complex balance as a parametrized system of polynomial {\em inequalities}. 
Given a reaction network with graph $G=(V,E)$ and positive rate constants $k \in \R^E_>$,
we show that whether~$k$ belongs to the disguised complex-balanced parameter locus~$\mathcal{K}^\text{dCB}$ is equivalent to the existence of positive state variables $x\in \R^n_>$
such that the vector of reaction rates
\[
v_k(x) = k \circ x^{Y_s I_{E,s}}
\]
belongs to the disguised complex-balanced flux cone~$\mathcal{C}^\text{dCB}$,
\begin{align*}
k \in \mathcal{K}^\text{dCB}
&\iff
\exists x \colon v_k(x) \in \mathcal{C}^\text{dCB} \\
&\iff 
\exists (x,\nu) \colon \nu = v_k(x) \wedge \nu \in \mathcal{C}^\text{dCB} ,
\end{align*}
cf.~Theorem~\ref{thm:KdCB_x}.
The first formulation exhibits disguised complex balance as a parametrized system of polynomial inequalities, where the vector of monomial terms $v_k(x)$ is constrained to lie in the polyhedral cone $\mathcal{C}^\text{dCB}$.
The second formulation defines a quantifier elimination problem with parameter $k$ and quantified variables $x$ and $\nu$. 
Although implicit in previous work on disguised complex balance, the characterization of Theorem~\ref{thm:KdCB_x} has, to our knowledge, not been stated explicitly.



Using the framework for positive algebraic geometry developed in~\cite{MuellerRegensburger2023a},
we reformulate this parametrized system of polynomial inequalities
as a system of linear inequalities (defining the disguised complex-balanced flux cone) and binomial equations (on that cone),
\[
k \in \mathcal{K}^\text{dCB} \iff \exists \nu \in \mathcal{C}^\text{dCB}_> \colon \nu^z=k^z , \, \forall z \in \mathcal{D} ,
\]
cf.~Theorem \ref{thm:KdCB_y}. Here, 
\[
\mathcal{D} = \ker \binom{Y_s I_{E,s}}{\ones^\trans}
\]
is the (monomial) dependency subspace,
recording the affine dependencies among the source complexes (equivalently, among the source monomials) of the reactions.
That is, as our main result, we eliminate the state variables $x$ from the problem, thereby bypassing one expensive quantifier elimination step in the computation of the disguised complex-balanced parameter locus.

In contrast, Boros et al.~\cite{BorosEtAl2026} compute the disguised complex-balanced parameter locus by performing quantifier elimination on the full system involving both state and flux variables, $x$ and~$\nu$.
At this point, we briefly clarify terminology. Complex balance is also referred to as vertex balance in the modern reaction network literature~\cite{CraciunMuellerPanteaYu2019}, and the (disguised) complex-balanced parameter locus is often called the (disguised) toric locus~\cite{BrustengaEtAl2022}. However, we avoid the term toric in this paper to prevent confusion with its established meaning in algebraic geometry.


Since our approach combines reaction network theory with recent developments in positive algebraic geometry, we present a largely self-contained treatment. In particular, we provide a streamlined proof of a key observation from Craciun et al.~\cite{CraciunJinYu2020} concerning dynamical equivalence to complex-balanced systems,
cf.~Theorem~\ref{thm:inclusion}.

Finally, we illustrate our results using the running example of Boros et al.~\cite{BorosEtAl2026}, a partially reversible cycle.
In particular, we obtain the disguised complex-balanced parameter locus analytically, after eliminating the state variables.


\subsubsection*{Notation}
We denote the positive (nonnegative) real numbers by $\R_>$ ($\R_\ge$).
Throughout the work, we use index notation: for a finite index set $I$,
we write $\R^I$ for the real vector space of vectors $x=(x_i)_{i \in I}$ with $x_i \in \R$,
and analogously we write $\R^I_>$ ($\R^I_\ge$).
For $I=\{1,\ldots,n\}$, we have the standard case $\R^I=\R^n$.
We write $x>0$ for $x \in \R^I_>$ and $x \ge 0$ for $x \in \R^I_\ge$.

For vectors $x,y \in \R^I$, we denote their componentwise (Hadamard) product by $x \circ y \in \R^I$.
For $x \in \R^I_>, \, y \in \R^I$, we define the (generalized) monomial $x^y = \prod_{i \in I} (x_i)^{y_i} \in \R_>$,
and for $x \in \R^I_>, \, Y \in \R^{I \times J}$, 
we define the vector of monomials $x^Y \in \R^J_>$ via $(x^Y)_j =x^{y(j)}$,
where $y(j)$ is the column of $Y$ with index $j \in J$.


\section{Basic definitions}




We recall basic notions for reaction networks with mass-action kinetics~\cite{adleman2014mathematics,voit2015150,yu2018mathematical} following~\cite{MuellerRegensburger2014}.
Note that we use index notation (rather than pure cardinality notation) in the definition of matrices,
thereby avoiding arbitrary orderings of the underlying sets.

\begin{definition} \label{def:net}
A reaction network $(G,y)$ 
is given by a simple directed graph $G=(V,E)$
with a finite set of vertices $V$ and edge set $E \subseteq V \times V$
together with an injective map ${y \colon V \to \R^n}$
(a matrix $Y \in \R^{n \times V}$).
\end{definition}

In terms of (chemical) reaction networks,
every vertex $i \in V$
is labeled with a {\em (stoichiometric) complex} $y(i) \in \R^n$,
representing a formal sum of $n$ species (e.g.\ $X_1, \dots, X_n$).
Every edge $(i \to j) \in E$ represents a {\em reaction} ${y(i) \to y(j)}$,
transforming the {\em source} complex $y(i)$ into the {\em target} complex $y(j)$.

The restriction of the map $y$ (the matrix $Y$) to the set of source vertices $V_s \subseteq V$ is denoted by $y_s \colon V_s \to \R^n$ (respectively, $Y_s \in \R^{n \times V_s}$).
If all components of the graph are strongly connected,
the network is called {\em weakly reversible}.

\begin{remark}
In the classical definition of a reaction network,
complexes and reactions are primary objects and induce the complex-reaction graph.
Reaction networks as in Definition~\ref{def:net} can also be viewed as Euclidean-embedded graphs~\cite{craciun2015toric,craciun2019polynomial}.
\end{remark} 

\begin{definition} 
A {\em mass-action system} $(G_k,y)$ is given by a reaction network $(G,y)$
and positive edge labels $k \in \R^E_>$.
\end{definition}

In terms of (chemical) reaction networks,
every edge $(i \to j) \in E$ (that is, every reaction $y(i) \to y(j)$) is labeled with a {\em rate constant} $k_{i \to j} > 0$.
\newdef{By a slight abuse of notation, we denote by $k \in \R^{V_s \times V}_\ge$ the nonnegative matrix
\[
k_{ij} = \begin{cases}
    k_{i \to j}, & \text{if } (i \to j) \in E , \\
    0, & \text{otherwise,}
\end{cases}
\]
corresponding to the positive vector  $k \in \R^E_>$.
}

The associated ODE system for the non-negative state variables {\em (concentrations)} $x \in \R^n_\ge$ 
is given by
\begin{equation} \label{eq:ode}
	\dd{x}{t} 
	= \sum_{(i \to j) \in E} k_{i \to j} \, x^{y(i)} \big( y(j)-y(i) \big) .
\end{equation}
The sum ranges over all reactions, 
and every summand is a product of the {\em reaction rate} $k_{i \to j} \, x^{y(i)}$, involving a monomial $x^{y} = \prod_{j=1}^n (x_j)^{y_j}$ determined by the source complex,
and the {\em reaction vector} $y(j)-y(i)$ given by both the source and target complexes.

Let $I_E \in \{-1,0,1\}^{V \times E}$ and $I_{E,s} \in \{0,1\}^{V_s \times E}$ be the incidence and source matrices of the digraph~$G$, respectively,
and $R_k = I_E \diag(k) \, I_{E,s}^\trans \in \R^{V \times V_s}$
be the rectangular ``Laplacian matrix''.
Further, 
recall the (source) complex matrices, $Y \in \R^{n \times V}$ and $Y_s \in \R^{n \times V_s}$,
and let $N = Y I_E \in \R^{n \times E}$ be the stoichiometric matrix
and $v_k(x) = k \circ x^{Y_s I_{E,s}} \in \R^E_\ge$ be the vector of reaction rates.
Then, 
the right-hand-side of \eqref{eq:ode} can be written in matrix form and decomposed into stoichiometric and graphical contributions,
\begin{equation} \label{eq:ode:matrices}
\begin{aligned}
	\dd{x}{t}
    &= \underbrace{Y I_E}_N \underbrace{\left( k \circ x^{Y_s I_{E,s}} \right)}_{v_k(x)}
	= Y \underbrace{I_E \diag(k) \, I_{E,s}^\trans}_{R_k} x^{Y_s} \\
    &= \Gamma_k \, x^{Y_s} .
    \end{aligned}
\end{equation}
\newdef{In the last step,
we have introduced the {\em kinetic matrix} 
\[
\Gamma_k = Y R_k \in \R^{n \times V_s} ,
\]
in analogy to the stoichiometric matrix $N = Y I_E$.
While the colums of the stoichiometric matrix (indexed by edges) are the reaction vectors,
the columns of the kinetic matrix (indexed by source vertices) are weighted sums of reaction vectors,
\[
\gamma^i_k = \sum_{(i \to j) \in E} k_{i \to j} \left( y(j)-y(i) \right) \in \R^n ,
\quad \text{for } i \in V_s ,
\]
and we call them {\em source vectors}.
If $\gamma^i_k \neq 0$, we call the source vertex $i \in V_s$ --- as well as the source complex $y(i)$ --- {\em active};
otherwise, we call it {\em inactive}.}

The change over time lies in the {\em kinetic subspace} $K = \im (Y R_k)$
and further in the stoichiometric subspace {\em stoichiometric subspace} $S = \im (Y I_E)$, 
\[
\dd{x}{t} \in K \subseteq S .
\]
Hence, trajectories are confined to cosets of $K$ and $S$, respectively,
that is, $x(t) \in x(0)+K \subseteq x(0) + S$.
For positive $x' \in \R^n_>$, the sets $(x'+K) \subseteq (x'+S) \cap \R^n_>$ are called {\em kinetic} and {\em stoichiometric compatibility classes}, respectively.
It is well known that $K=S$ if the network is weakly reversible; cf.~\cite{FeinbergHorn1977}.

\begin{remark}
Traditionally,
one uses a source matrix $\mathcal{I}_{E,s} \in \{0,1\}^{V \times E}$ that involves all vertices (not just the source vertices),
and one obtains
\[
\dd{x}{t} = Y I_E \left( k \circ x^{Y \mathcal{I}_{E,s}} \right) 
= Y \underbrace{I_E \diag(k) \, \mathcal{I}_{E,s}^\trans}_{\mathcal{L}_k} x^{Y}
\]
with the (square) Laplacian matrix $\mathcal{L}_k \in \R^{V \times V}$.
This formulation can be misleading
since columns of $\mathcal{L}_k$ corresponding to non-source vertices are zero,
and, after multiplication, non-source monomials do not appear in $\dd{x}{t} = Y \mathcal{L}_k \, x^{Y}$ for any~$k$.
Indeed,
$R_k$ arises from $\mathcal{L}_k$ by deleting zero columns,
and $\im R_k = \im \mathcal{L}_k$. 
\end{remark}

\subsubsection*{Complex balance}

\begin{definition}
Let $(G_k,y)$ be a mass-action system.
A positive equilibrium $x \in \R^n_>$ of the associated ODE system~\eqref{eq:ode:matrices} is called {\em complex-balanced} if
\[
I_E \, v_k(x) = R_k \, x^{Y_s} = \mathcal{L}_k \, x^Y = 0 .
\]
\end{definition}
If there exists a positive complex-balanced equilibrium, 
then all equilibria are complex-balanced~\cite{HornJackson1972}, and $(G_k,y)$ is called complex-balanced.
If $(G_k,y)$ is complex-balanced, then $(G,y)$ is weakly reversible~\cite{Horn1972} and hence $V=V_s$.

\newdef{For positive $x \in \R^n_>$, we denote by $v \in \R^{V \times V}_\ge$ the nonnegative matrix
\[
v_{ij} = k_{ij} \, x^{y(i)}, \quad \text{for } i,j \in V ,
\]
corresponding to the positive vector $v_k(x) \in \R^E_>$.
Then the action of the incidence matrix is given by
\[
\left( I_E \, v_k(x) \right)_i = \sum_{j \in V} v_{ij} - \sum_{j \in V} v_{ji} , \quad \text{for } i \in V ,
\]
and complex balance is equivalent to $\sum_{j \in V} v_{ij} = \sum_{j \in V} v_{ji}$ for all $i \in V$.
}

\begin{definition}
Let $(G,y)$ with $G=(V,E)$ be a reaction network.
Then, its {\em complex-balanced parameter locus} is given by
\[
\mathcal{K}^\text{CB}_{(G,y)} = \{ k \in \R^E_> \mid \exists x \in \R^n_> \colon I_E \, v_k(x) = 0 \} .
\]
\end{definition}


\section{Results}

A reaction network with parameters (a mass-action system) need not be complex-balanced, 
but it may have a dynamically equal realization that is complex-balanced.

\subsection{Dynamical equality}

\begin{definition} \label{def:dynequ}
Let $(G_k,y)$ with $G=(V,E)$ and
$(G'_{k'},y')$ with $G'=(V',E')$ be mass-action systems.
Then, $(G_k,y')$ and $(G'_{k'},y')$ are {\em dynamically equal} if
\[
\Gamma_{(G_k,y)} \, x^{Y_s} = \Gamma_{(G'_{k'},y')} \, x^{Y'_s}
\]
for all $x \in \R^n_>$.
We write $(G_k,y) \de (G'_{k'},y')$.
\end{definition}

Let two mass-action systems be dynamically equal.
If one system contains an {\em active} source vertex---with corresponding source complex and {\em nonzero} source vector,
then the other system must contain the same source complex---for some active source vertex. 
However, both systems may contain additional source vertices/complexes as long as they are inactive (and hence do not affect the dynamics).

\subsubsection*{Dynamical equality and complex balance}

In Theorem~\ref{thm:CJY2020} below, we identify all relevant complex-balanced mass-action systems $(G'_{k'},y')$
that are dynamically equal to a given mass-action system $(G_k,y)$.
By definition,
for every active source vertex $i \in V_s$ of the given system,
the dynamically equal system must also contain the source complex $y(i)$.
Without loss of generality, 
we may assume $i \in V'_s$ and $y'(i) = y(i)$.
As it turns out, we do not have to consider inactive sources in dynamically equal, complex-balanced systems,
that is, $V'_s \subseteq V_s$.
This was observed in~\cite[Section~4]{CraciunJinYu2020}.

We first formulate a result corresponding to~\cite[Lemma~4.3, Theorems~4.4 and~4.5]{CraciunJinYu2020} that uses essentially the same construction (a redirection of ``fluxes''), but avoids extra definitions such as flux systems, flux equivalence, potential, etc.
To ensure that the paper is self-contained, we provide a proof,
which uses ``fluxes'' only as a name for reaction rates (monomial terms).

\begin{theorem} \label{thm:inclusion}
Let $(G_k,y)$ with $G=(V,E)$ be a complex-balanced mass-action system.
Then, there exists a dynamically equal, complex-balanced mass-action system $(G'_{k'},y')$ with $G'=(V',E')$
such that $V' = V_a$, $y' = y|_{V'}$,
where $V_a \subseteq V_s$ is the set of active source vertices of $(G_k,y)$.
\end{theorem}
\begin{proof}
We proceed by iteratively eliminating inactive source vertices. 
Thus, it suffices to consider a single step. 
Indeed, we show that eliminating one inactive source vertex produces a dynamically equal, complex-balanced mass-action system.
Since the number of inactive source vertices is finite, repeated application of this construction yields the desired system.

Let $i^* \in V \, (= V_s)$ be an inactive source vertex of a complex-balanced mass-action system $(G_k,y)$ with $G=(V,E)$.
Further, let $V' = V \setminus \{i^* \}$ be the set of remaining (active or inactive) source vertices.
Then,
\[
0 = \gamma^{i^*}_k
= \sum_{(i^* \to j) \in E} k_{i^* \to j} \left( y(j)-y(i^*) \right)
= \sum_{j \in V'} k_{i^*j} \left( y(j)-y(i^*) \right) ,
\]
and hence the inactive source complex is the barycenter of its target complexes,
\[
y(i^*) = \sum_{j \in V'} \lambda_{i^*j} \, y(j) 
\quad
\text{with convex weights}
\quad
\lambda_{i^*j} = \frac{k_{i^*j}}{\sum_{j' \in V'} k_{i^*j'}} ,
\]
having the normalization $\sum_{j \in V'} \lambda_{i^*j} = 1$.
Now, 
let $i \in V'$. Then,
\begin{align*}
\gamma^i_k &= \sum_{j \in V} k_{ij} \left( y(j)-y(i) \right) \\
&= \sum_{j \in V'} k_{ij} \left( y(j)-y(i) \right) + k_{ii^*} \left( y(i^*)-y(i) \right) \\
&= \sum_{j \in V'} k_{ij} \left( y(j)-y(i) \right) + k_{ii^*} \sum_{j \in V'} \lambda_{i^*j} \left( y(j)-y(i) \right) \\
&= \sum_{j \in V'} \left( k_{ij} + k_{ii^*} \lambda_{i^*j} \right) \left( y(j)-y(i) \right) ,
\end{align*}
where we have used the barycenter formula for the inactive source complex and the normalization of convex weights in the third step.
We therefore define a mass-action system on $V'$ 
with $y' = y|_{V'}$, rate constants
\[
k'_{ij} = k_{ij} + k_{ii^*} \lambda_{i^*j} , \quad \text{for } i,j \in V' ,
\]
and hence edge set
\[
E' = \{i \to j \mid i,j \in V', \, k'_{ij}>0 \} .
\]
Dynamical equality holds by definition. Indeed, for $i \in V'$,
\[
\gamma^{i}_{k'} = \sum_{j \in V'} k'_{ij} \left( y(j)-y(i) \right) = \gamma^i_k .
\]
For the original system,
we denote by $v \in \R^{V \times V}_\ge$ the nonnegative matrix $v_{ij} = k_{ij} \, x^{y(i)}$, for $i,j \in V$,
and complex balance is given by $\sum_{j \in V} v_{ij} = \sum_{j \in V} v_{ji}$ for all $i \in V$.
(Since all graphs are simple, this is equivalent to $\sum_{j \in V\setminus\{i\}} v_{ij} = \sum_{j \in V\setminus\{i\}} v_{ji}$.)
Correspondingly, for the new system, we denote by $v' \in \R^{V' \times V'}_\ge$ the nonnegative matrix
\[
v'_{ij} = k'_{ij} \, x^{y(i)}, \quad \text{for } i,j \in V' .
\]
Clearly,
\begin{equation} \label{eq:dv} \tag{$\ast$}
v'_{ij} = \left( k_{ij} + k_{ii^*} \lambda_{i^*j} \right) x^{y(i)}
= v_{ij} + v_{ii^*} \lambda_{i^*j} .
\end{equation}
By assumption, the original system is complex-balanced.
It remains to show that also the new system (after elimination of $i^*$) is complex-balanced.
Indeed, we show that the eliminated in/out-fluxes from/to~$i^*$ are compensated by new fluxes within $V'$.

First, for $i \in V'$, the eliminated out-flux $v_{ii^*}$ to~$i^*$ is redirected to $j \in V'$ according to the weights $\lambda_{i^*j}$.
Explicitly, by Equation~\eqref{eq:dv},
\[
\sum_{j \in V'} v'_{ij} - v_{ij} = \sum_{j \in V'} v_{ii^*} \lambda_{i^*j} = v_{ii^*} ,
\]
where we have used the normalization of convex weights in the last step.

Second, for $i \in V'$, the eliminated in-flux $v_{i^*i}$ from~$i^*$ is also recovered.
Observe
\[
\lambda_{i^*i} = \frac{k_{i^*i}}{\sum_{j' \in V'} k_{i^*j'}}
= \frac{v_{i^*i}}{\sum_{j' \in V'} v_{i^*j'}} = \frac{v_{i^*i}}{\sum_{j' \in V'} v_{j'i^*}} ,
\]
where we used complex balance for $i^*$ in the last step.
Now,
\[
\sum_{j \in V'} v'_{ji} - v_{ji} = \sum_{j \in V'} v_{ji^*} \lambda_{i^*i} 
= \sum_{j \in V'} v_{ji^*} \frac{v_{i^*i}}{\sum_{j' \in V'} v_{j'i^*}}
= v_{i^*i}.
\]
Therefore, the total in/out-fluxes at every vertex $i \in V'$ are preserved. 
Since the original system is complex-balanced, also the new system is complex-balanced.

We iteratively eliminate all inactive source vertices until we obtain a dynamically equal, complex-balanced system with $V'=V_a$.
\end{proof}

The next result is an immediate consequence of Theorem~\ref{thm:inclusion}. Indeed, every dynamically equal, complex-balanced realization of a given mass-action system induces a dynamically equal, complex-balanced realization that uses only source complexes (equivalently, monomials) from the original system.

\begin{theorem}[\cite{CraciunJinYu2020}, Theorem~4.7] \label{thm:CJY2020}
Let $(G_k,y)$ with $G=(V,E)$ be a mass-action system.
Then, there exists a dynamically equal, complex-balanced mass-action system 
if and only if there exists a dynamically equal, complex-balanced mass-action system $(G'_{k'},y')$ with $G'=(V',E')$
such that $V' = V'_s \subseteq V_s \subseteq V$ and $y' = y|_{V'}$.
\end{theorem}

\subsection{Disguised complex balance}

As mentioned above,
a mass-action system (a reaction network with rate constants) need not be complex-balanced, 
but it may have a dynamically equal realization that is complex-balanced (for other rate constants).

\begin{definition}
Let $(G,y)$ with $G=(V,E)$ be a reaction network. Then, its {\em disguised complex-balanced parameter locus} is given by 
\[
\mathcal{K}^\text{dCB}_{(G,y)} = 
\{ k \in \R^E_> \mid 
\exists (G',y') ,\, \exists k' \in \mathcal{K}^\text{CB}_{(G',y')} \text{ such that } 
(G_k,y) \de (G'_{k'},y') \} .
\]
\end{definition}

By complex balance, it suffices to consider weakly reversible graphs $G'$.
By Theorem~\ref{thm:CJY2020},
it suffices to consider graphs $G' = (V',E')$ with $V'_s \subseteq V_s$. 
For simplicity, we consider all subgraphs $G'$ of the complete simple directed graph $G^\text{com}_{V_s}$ on $V_s$.
(However, as noted in \cite[Section~3]{BorosEtAl2026}, the complete graph need not be minimal for the characterization of the disguised locus.)

\subsubsection*{From graphs to vectors (of edge labels)}

Let $G^\text{com}_{V_s} = (V_s, E^\text{com})$. 
Then every nonnegative vector $\bar k \in \R^{E^\text{com}}_\ge$ induces a graph $G'=(V',E')$ with 
\[
V' = \{ i \mid \exists j \text{ s.t. } \bar k_{i \to j}>0 \text{ or } \bar k_{j \to i}>0 \}
\quad \text{and} \quad
E' = \{ i \to j \mid \bar k_{i \to j}>0 \} ,
\]
a corresponding positive vector $k' \in \R^{E'}_>$,
and a mass-action system $(G'_{k'},y')$ with $y'=y|_{V'}$.
For $x \in \R^n_>$, the complex-balance equation takes the form $I_{E'} \, v_{k'}(x) = 0$ with $I_{E'} \in \{-1,0,1\}^{V' \times E'}$ and $v_{k'}(x) \in \R^{E'}_>$.
Notably,
the kinetic matrix $\Gamma_{(G'_{k'},y')} \in \R^{n \times V'_s}$
may have fewer columns than the kinetic matrix 
\[
\Gamma_k := \Gamma_{(G_k,y)} \in \R^{n \times V_s}
\]
of the mass-action system $(G_k,y)$, since $V'_s \subseteq V_s$ (that is, due to ``unused'' source vertices).
This construction provides a graph-theoretic interpretation of a nonnegative vector $\bar k \in \R^{E^\text{com}}_\ge$. 

For the analysis below, it is convenient to work directly on the complete graph $G^\text{com}_{V_s}$.
Thus, we symbolically compute once and for all the complex-balance equation $I_{E^\text{com}} \, v_{\bar k} (x) = 0$ and the kinetic matrix 
\[
\Gamma_{\bar k} := \Gamma_{((G^\text{com}_{V_s})_{\bar k},y_s)} \in \R^{n \times V_s} ,
\]
where $\bar k \in \R^{E^\text{com}}_>$.
Subsequently, we allow $\bar k \in \R^{E^\text{com}}_\ge$ such that vanishing entries automatically encode the corresponding subgraphs.
In this ambient representation, the kinetic matrices 
$\Gamma_k$ and $\Gamma_{\bar k}$
have the same number of columns,
since vanishing entries of $\bar k$ automatically produce zero columns corresponding to unused source vertices.
Consequently,
dynamical equality $(G_k,y) \de (G_{k'},y')$ is equivalent to the matrix equality
\[
\Gamma_k = \Gamma_{\bar k} .
\]

The following result summarizes the preceding constructions and provides an equivalent formulation of the disguised complex-balanced parameter locus in terms of the complete graph alone.

\begin{proposition} \label{prop:complete}
Let $(G,y)$ with $G=(V,E)$ be a reaction network,
and recall $G^\textnormal{com}_{V_s} = (V_s, E^\textnormal{com})$.
Then,
\[
\mathcal{K}^\textnormal{dCB}_{(G,y)} = 
\{ k \in \R^E_> \mid 
\exists x \in \R^n_> , \, 
\exists \bar k \in \R^{E^\textnormal{com}}_\ge
\textnormal{ such that } 
I_{E^\textnormal{com}} \, v_{\bar k} (x) = 0 
\textnormal{ and } 
\Gamma_k = \Gamma_{\bar k}
\} .
\]   
\end{proposition}

\subsection{From disguised complex balance to positive algebraic geometry}

We aim to reformulate the disguised complex-balanced parameter locus
as a system of linear constraints for the vector of monomial terms.
To this end, we introduce nonnegative variables $\nu \in \R^E_\ge$ and $\bar\nu \in \R^{E^\text{com}}_\ge$
such that
\[
\nu= v_k(x)
\quad \text{and} \quad
\bar\nu = v_{\bar k} (x) .
\]
Then, (i) the complex-balance equation becomes a linear constraint, $I_{E^\text{com}} \, \bar\nu = 0$.
As we will show, (ii) dynamical equality becomes $\Gamma_\nu= \Gamma_{\bar\nu}$,
and (iii) $\bar k$ can be eliminated subsequently.
(Note that we allow $\nu\in \R^E_\ge$, whereas $v_k(x) \in \R^E_>$.
This relaxation will allow us to use polyhedral geometry.)

First, we consider ``kinetic matrices'' of ``mass-action systems'', arising from a reaction network and {\em nonnegative} rate constants, such as $\Gamma_{\bar k}$ for $G^\text{com}_{V_s}$.
(For positive parameters, these are the ``standard'' kinetic matrices of mass-action systems.)

\begin{lemma} \label{lem:IEs}
Let $(G,y)$ with $G=(V,E)$ be a reaction network, $k \in \R^E_\ge$ be nonnegative parameters,
$\Gamma_k = Y I_E \diag(k) \, I_{E,s}^\trans \in \R^{n \times V_s}$,
and $v_k(x) = k \circ x^{Y_s I_{E,s}} \in \R^E_\ge$, for $x \in \R^n_>$.
Then,
\[
\Gamma_k \diag (x^{Y_s}) = \Gamma_{v_k(x)} ,
\]
where $\Gamma_{v_k(x)}$ is obtained by replacing $k$ with $v_k(x)$ in the definition of $\Gamma_k$.
\end{lemma}
\begin{proof}
Recall that every column (indexed by an edge) of the source matrix $I_{E,s} \in \{0,1\}^{V_s \times E}$ contains exactly one nonzero entry, namely in the row corresponding to the source vertex of the edge.
It is easy to see that, for $z \in \R^{V_s}$,
\[
\diag(z) \, I_{E,s}
=
I_{E,s} \, \diag(z^{I_{E,s}}) .
\]
Just observe that
multiplying $I_{E,s}$ with $\diag(z)$ from the left,
that is, multiplying every row with a component of $z$,
yields the same result as multiplying with $\diag(z^{I_{E,s}})$ from the right,
where components of $z$ have been ``duplicated'' accordingly.
Using (the transpose of) this identity,
\begin{align*}
\Gamma_k \diag(x^{Y_s})
&=
Y I_E \diag(k) \, I_{E,s}^\trans \diag(x^{Y_s}) \\
&=
Y I_E \diag(k)\diag(x^{Y_s I_{E,s}}) I_{E,s}^\trans \\
&=
Y I_E \diag(v_k(x)) \, I_{E,s}^\trans \\
&=
\Gamma_{v_k(x)}.
\end{align*}
\end{proof}

Next, we address equality between the kinetic matrix of a mass-action system and the ``kinetic matrix'' of a corresponding reaction network (such as the complete graph on the source vertices of the mass-action system) and nonnegative parameters.

\begin{lemma} \label{lem:Gammas}
Let $(G_k,y)$ with $G=(V,E)$ be a mass-action system and $x \in \R^n_>$.
Further,
let $(G',y')$ with $G'=(V',E')$ be a reaction network with $V'_s=V_s$ and $y'=y_s$, $k' \in \R^{E'}_\ge$ be nonnegative parameters,
$\Gamma_{k'} = Y I_{E'} \diag(k') \, I_{E',s}^\trans \in \R^{n \times V_s}$,
and $v_{k'}(x) = k' \circ x^{Y_s I_{E',s}} \in \R^{E'}_\ge$.
Then,
\[
\Gamma_k = \Gamma_{k'}
\quad \Longleftrightarrow \quad
\Gamma_{v_k(x)} = \Gamma_{v_{k'}(x)} .
\]
\end{lemma}
\begin{proof}
By Lemma~\ref{lem:IEs},
\[
\Gamma_k\diag(x^{Y_s})
=
\Gamma_{v_k(x)}
\quad\text{and}\quad
\Gamma_{k'}\diag(x^{Y_s})
=
\Gamma_{v_{k'}(x)} .
\]
Since $\diag(x^{Y_s})$ is invertible for $x \in \R^n_>$,
the claim follows.
\end{proof}

As announced, 
we can now reformulate the disguised complex-balanced parameter locus.
Starting from Proposition~\ref{prop:complete},
we (i) introduce new variables
and (ii) use Lemma~\ref{lem:Gammas},
\begin{alignat*}{3}
\mathcal{K}^\text{dCB}_{(G,y)} &= 
\{ k \in \R^E_> \mid 
\exists x \in \R^n_> , \; &&
\exists \bar k \in \R^{E^\text{com}}_\ge 
\colon 
I_{E^\text{com}} \, v_{\bar k} (x) = 0 \wedge 
\Gamma_k = \Gamma_{\bar k} \} \\
&\stackrel{\textrm{(i)}}{=} \{ k \in \R^E_> \mid 
\exists x \in \R^n_> , &&
\exists \bar k \in \R^{E^\text{com}}_\ge , \,
\exists \nu\in \R^E_\ge , \, 
\exists \bar\nu \in \R^{E^\text{com}}_\ge \colon \\
&&&
\nu= v_k(x) \wedge
\bar\nu = v_{\bar k}(x) \wedge
I_{E^\text{com}} \, v_{\bar k} (x) = 0 \wedge 
\Gamma_k = \Gamma_{\bar k} \} \\
&\stackrel{\textrm{(ii)}}{=} \{ k \in \R^E_> \mid 
\exists x \in \R^n_> , &&
\exists \bar k \in \R^{E^\text{com}}_\ge , \,
\exists \nu\in \R^E_\ge , \, 
\exists \bar\nu \in \R^{E^\text{com}}_\ge \colon \\
&&&
\nu= v_k(x) \wedge
\bar\nu = v_{\bar k}(x) \wedge
I_{E^\text{com}} \, \bar\nu = 0 \wedge 
\Gamma_\nu= \Gamma_{\bar\nu} \} \\
&\stackrel{\textrm{(iii)}}{=} \{ k \in \R^E_> \mid 
\exists x \in \R^n_> , &&
\exists \nu\in \R^E_\ge , \, 
\exists \bar\nu \in \R^{E^\text{com}}_\ge \colon 
\nu= v_k(x) \wedge
I_{E^\text{com}} \, \bar\nu = 0 \wedge 
\Gamma_\nu= \Gamma_{\bar\nu} \} .
\end{alignat*}
Finally, in (iii), we have eliminated $\bar k$,
since the only constraint involving $\bar k$, 
\[
\bar\nu = v_{\bar k}(x) = \bar k \circ x^{Y_s I_{E^\text{com},s}} ,
\]
determines a unique $\bar k$ for any $x$ and $\bar\nu$.

The following result summarizes the preceding variable introduction/elimination process and provides an equivalent formulation
of the disguised complex-balanced parameter locus. 

\begin{proposition}
Let $(G,y)$ with $G=(V,E)$ be a reaction network,
and recall $G^\textnormal{com}_{V_s} = (V_s, E^\textnormal{com})$.
Then,
\[
\mathcal{K}^\textnormal{dCB}_{(G,y)} = 
\{ k \in \R^E_> \mid 
\exists x \in \R^n_> , \,
\exists \nu\in \R^E_\ge , \, 
\exists \bar\nu \in \R^{E^\textnormal{com}}_\ge
\textnormal{ such that } 
\nu= v_k(x) , \,
I_{E^\textnormal{com}} \, \bar\nu = 0 , \,
\textnormal{and } 
\Gamma_\nu= \Gamma_{\bar\nu} 
\} .
\]   
\end{proposition}

Finally, we eliminate also the auxiliary variables $\nu, \bar\nu$ from the linear equations for complex balance and dynamical equality.
First, we introduce the polyhedral cone
\[
C^\text{s} = \{ (\nu,\bar\nu) \in \R^E_\ge \times \R^{E^\text{com}}_\ge \mid 
I_{E^\text{com}} \, \bar\nu = 0 
\text{ and } 
\Gamma_\nu= \Gamma_{\bar\nu}
\} ,
\]
which is an ``s-cone''~\cite{MuellerRegensburger2016} (arising from a linear subspace and nonnegativity constraints).
Then, its projection to the $\nu$-variables is another polyhedral cone,
which is a crucial geometric object for disguised complex-balance.

\begin{definition} \label{def:CdCB}
Let $(G,y)$ with $G=(V,E)$ be a reaction network,
and recall $G^\text{com}_{V_s} = (V_s, E^\text{com})$.
The {\em disguised complex-balanced flux cone} is given by
\[
\mathcal{C}^\text{dCB}_{(G,y)} 
= \{ \nu\in \R^E_\ge \mid 
\exists \bar\nu \in \R^{E^\text{com}}_\ge
\text{ such that } 
I_{E^\text{com}} \, \bar\nu = 0 
\text{ and } 
\Gamma_\nu= \Gamma_{\bar\nu} 
\} .
\]
\end{definition}
By complex balance and dynamical equivalence,
$\mathcal{C}^\text{dCB}_{(G,y)}$ is a subcone of the flux cone 
\[
\mathcal{C}^\text{flx}_{(G,y)} = \{ \nu\in \R^E_\ge \mid N \nu= 0 \} ,
\]
which motivates its name.

Now,
\begin{equation*} 
\begin{aligned}
\mathcal{K}^\text{dCB}_{(G,y)} &= \{ k \in \R^E_> \mid 
\exists x \in \R^n_> ,
\exists \nu\in \R^E_\ge , \, 
\exists \bar\nu \in \R^{E^\text{com}}_\ge \colon 
\nu= v_k(x) \wedge
I_{E^\text{com}} \, \bar\nu = 0 \wedge 
\Gamma_\nu= \Gamma_{\bar\nu} \} \\
&= \{ k \in \R^E_> \mid 
\exists x \in \R^n_> ,
\exists \nu\in \R^E_\ge \colon 
\nu= v_k(x) \wedge
\nu\in \mathcal{C}^\text{dCB}_{(G,y)} \} \\
&= \{ k \in \R^E_> \mid 
\exists x \in \R^n_> \colon 
v_k(x) \in \mathcal{C}^\text{dCB}_{(G,y)}
\} .
\end{aligned}
\end{equation*}
Since $v_k(x) \in \mathbb{R}^E_{>}$ for every $k \in \mathbb{R}^E_{>}$ and $x \in \mathbb{R}^n_{>}$, the membership condition is equivalently restricted to the positive part of the cone.
Hence, we can state the desired reformulation of the disguised complex-balanced parameter locus in two equivalent ways.
In addition to the abstract formulations, we also state the locus as an explicit quantifier elimination problem.

\begin{theorem}[Polyhedral geometry of $\mathcal{K}^\textnormal{dCB}_{(G,y)}$] \label{thm:KdCB_x}
Let $(G,y)$ with $G=(V,E)$ be a reaction network
and $\mathcal{C} := \mathcal{C}^\textnormal{dCB}_{(G,y)}$.
Then,
\begin{align*}
\mathcal{K}^\textnormal{dCB}_{(G,y)} 
&= \{ k \in \R^E_> \mid 
\exists x \in \R^n_> 
\textnormal{ such that } 
v_k(x) \in \mathcal{C} \} \\
&= \{ k \in \R^E_> \mid 
\exists x \in \R^n_> 
\textnormal{ such that } 
v_k(x) \in \mathcal{C}_> \} .
\end{align*}
Stating all variables explicitly,
\begin{align*}
\mathcal{K}^\textnormal{dCB}_{(G,y)} 
&= \{ k \in \R^E_> \mid 
\exists x \in \R^n_> ,
\exists \nu\in \R^E_> \textnormal{ such that } 
\nu= v_k(x) \textnormal{ and }
\nu\in \mathcal{C} \} .
\end{align*}
\end{theorem}

Theorem~\ref{thm:KdCB_x} expresses the disguised complex-balanced parameter locus as a system of generalized polynomial inequalities~\cite{MuellerRegensburger2023a}.
Indeed, with $\mathcal{C} := \mathcal{C}^\textnormal{dCB}_{(G,y)}$, the condition
\[
v_k(x) = \left( k \circ x^{Y_s I_{E,s}} \right) \in \mathcal{C}_>
\]
is precisely of the form 
\[
\left( c \circ x^B \right) \in C ,
\]
with positive parameter vector $c=k$, exponent matrix $B=Y_s I_{E,s}$, and positive cone $C=\mathcal{C}_>$.
Hence, the disguised complex-balanced parameter locus fits into the framework for generalized polynomial inequalities developed in~\cite{MuellerRegensburger2023a}.

\begin{remark}
As noted above, the complete graph need not be minimal for the characterization of disguised complex balance.
A natural criterion for retaining an edge $e \in E^\text{com}$ is that
$\bar\nu_e \not\equiv 0$ on the s-cone $C^\text{s}$.
Since Theorem~\ref{thm:KdCB_x} effectively depends only on the positive part of $\mathcal{C}^\text{dCB}_{(G,y)}$,
a more restrictive criterion asks whether there exists a feasible pair
$(\nu,\bar\nu)\in C^\text{s}$ such that
\[
\nu>0
\quad\text{and}\quad
\bar\nu_e>0.
\]
This leads to the graph of relevant edges $G^\text{rel}_{V_s} = (V_s,E^\text{rel})$, referred to as the maximal graph~$G^\text{max}$ in~\cite{BorosEtAl2026}.
\end{remark}

\subsection{Positive algebraic geometry}

We briefly summarize the main results from~\cite{MuellerRegensburger2023a} on generalized polynomial inequalities 
that are needed below.

\begin{theorem}[\cite{MuellerRegensburger2023a}, Theorem 5 and Proposition 6] \label{thm:previous}
Consider the parametrized system of generalized polynomial inequalities $(c \circ x^B) \in C$ for positive variables $x \in \R^n_>$,
given by a real exponent matrix $B \in \R^{n \times m}$, 
a positive parameter vector $c \in \R^m_>$,
and positive coefficient cone $C \subseteq \R^m_>$.
Then, the solution set $\mathcal{X}_c = \{ x \in \R^n_> \mid (c \circ x^B) \in C \}$ can be written as
\[
\mathcal{X}_c = \{ (y \, \circ\, c^{-1})^E \mid y \in \mathcal{Y}_c \} \circ \e^{L^\perp} ,
\]
where
\[
\mathcal{Y}_c = \{ y \in P \mid 
y^z = c^z \text{ for all } z \in D \} 
\]
is the solution set on the coefficient polytope $P$,
further $D$ and $L$ are the monomial dependency and difference subspaces, respectively,
and $E$ is a generalized inverse.

In particular, there is a bijection between $\mathcal{X}_c / \e^{L^\perp}$ and $\mathcal{Y}_c$.
\end{theorem}

The relevant geometric objects, namely the coefficient polytope \(P\), the monomial dependency and difference subspaces \(D\) and \(L\), respectively, as well as a generalized inverse \(E\), are introduced in~\cite[Section~2]{MuellerRegensburger2023a}. 

For existence questions, only the binomial equations on $P$, determined by $D$ are relevant.
The exponential manifold induced by $L$, and the reconstruction via $E$ are not required.


For notational simplicity, we assume that $C$ is not a direct product of cones.
Then,
\[
P = C \cap \Delta
\]
where $\Delta$ is the simplex in $\R^m_\ge$,
and
\[
D = \ker \mathcal B
\quad \text{with} \quad
\mathcal{B} =
\begin{pmatrix}
B \\ \ones^\trans
\end{pmatrix} .
\]

The normalization via $P$ reduces dimensions and eliminates a ``radial'' symmetry in $C$.
Indeed, for $y \in C$, $z \in D$, and $\lambda>0$, we have
$(\lambda y)^z = \lambda^{\ones^\trans \!z} \, y^z = y^z$, since $\ones^\trans z = 0$ for $z \in D$.

Theorem~\ref{thm:previous} implies the following
existence result.

\begin{corollary} \label{cor:previous}
Consider the parametrized system of generalized polynomial inequalities $(c \circ x^B) \in C$,
and the solution sets
\[
\mathcal{X}_c = \{ x \in \R^n_> \mid (c \circ x^B) \in C \}
\quad \text{and} \quad
\mathcal{Y}_c = \{ y \in P \mid 
y^z = c^z \text{ for all } z \in D \} ,
\]
where $P$ is the coefficient polytope and $D$ the monomial dependency subspace.
Then,
\[
\mathcal{X}_c \neq \emptyset
\quad \iff \quad
\mathcal{Y}_c \neq \emptyset .
\]
\end{corollary}

{\bf Notation.}
In the abstract setting of generalized polynomial inequalities,
we have used the symbol $y$ for elements of $\mathcal{Y}_c$,
for better readability.
In the setting of disguised complex balance,
we have to avoid a symbol clash with $y,Y$ (and $y_s,Y_s$), denoting complexes.

\subsubsection*{Main result}

We apply Corollary~\ref{cor:previous}, stated for the abstract system $(c \circ x^B) \in C$, to the characterization of the disguised complex-balanced parameter locus in Theorem~\ref{thm:KdCB_x}.
There, $v_k(x) = k \circ x^{Y_s I_{E,s}}$, so that $c=k$ and $B=Y_s I_{E,s}$; further $C=\mathcal{C}_>$ with $\mathcal{C} = \mathcal{C}^\text{dCB}_{(G,y)}$.
Hence,
\begin{align*}
\mathcal{K}^\textnormal{dCB}_{(G,y)} 
&= \{ k \in \R^E_> \mid 
\exists x \in \R^n_> 
\colon
v_k(x) \in \mathcal{C}_> \} \\
&= \{ k \in \R^E_> \mid 
\exists x \in \R^n_> 
\colon
( k \circ x^{Y_s I_{E,s}} ) \in \mathcal{C}_> \} \\
&= \{ k \in \R^E_> \mid
\exists \nu \in \mathcal{P} \colon \nu^z = k^z , \forall z \in \mathcal{D} \} ,
\end{align*}
where $\mathcal{P} = \mathcal{C}_> \cap \Delta$ and $\mathcal{D} = \ker B \cap \ker \ones^\trans$.
We have shown the main result of this work.

\begin{theorem} \label{thm:KdCB_y}
Let $(G,y)$ with $G=(V,E)$ be a reaction network, \[
\mathcal{P} = \left( \mathcal{C}^\textnormal{dCB}_{(G,y)}\right)_> \cap \Delta
\quad \text{and} \quad
\mathcal{D} = \ker \binom{Y_s I_{E,s}}{\ones^\trans} .
\]
Then,
\[
\mathcal{K}^\textnormal{dCB}_{(G,y)} = 
\{ k \in \R^E_> \mid 
\exists \nu \in \mathcal{P}
\textnormal{ such that } 
\nu^z=k^z
\textnormal{ for all } 
z \in \mathcal{D}
\} .
\]
\end{theorem}

The characterization of $\mathcal{K}^\text{dCB}_{(G,y)}$ given in Theorem~\ref{thm:KdCB_x} can be viewed as a quantifier elimination problem for the $n+|E|$ variables $x \in \R^n_>$ and $\nu \in \R^E_>$.
Applying the framework for generalized polynomial inequalities developed in~\cite{MuellerRegensburger2023a},
we eliminate the $n$ (state) variables $x \in \R^n_>$ and obtain an equivalent characterization in terms of the (flux) variables $\nu \in \R^E_>$ only, as stated in Theorem~\ref{thm:KdCB_y}.


\section{Example}

We analyze the running example from~\cite{BorosEtAl2026}; see Figure~\ref{fig}.
The reaction network $(G,y)$ is given by the graph $G=(V,E)$ with $V=V_s=\{1,2,3,4\}$, $E=\{ 1 \to 2, 1 \to 4 , 2 \to 3, 3 \to 2, 3 \to 4, 4 \to 1 \}$ and the map $y=y_s \colon V \to \R^2_\ge$ with $y(1) = (0,0)^\trans, y(2) = (1,0)^\trans, y(3) = (1,1)^\trans, y(4) = (0,1)^\trans$.

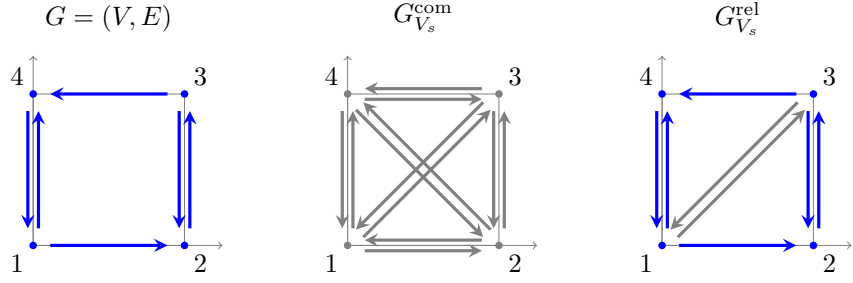
\begin{figure}[!ht]
\begin{center}
\begin{tikzpicture}[scale=2]
    \input{fig_run_G}
\end{tikzpicture}
\hspace{1cm}
\begin{tikzpicture}[scale=2]
    \input{fig_run_Gcom}
\end{tikzpicture}
\hspace{1cm}
\begin{tikzpicture}[scale=2]
    \input{fig_run_Grel}
\end{tikzpicture}
\caption{
Reaction network (as an embedded graph),
complete graph (on the source vertices),
and relevant graph for disguised complex balance}
\label{fig}
\end{center}
\end{figure}

We write dynamical equality of the mass-action systems $(G_k,y)$ and $((G^\text{com}_{V_s})_{\bar k},y_s)$ in terms of $\nu = v_k(x) \in \R^E_>$ and $\bar\nu = v_{\bar k}(x) \in \R^{E^\text{com}}_\ge$,
\begin{align*}
\Gamma_\nu &=
\begin{pmatrix}
\nu_{12} & 0        & - \nu_{34} & 0 \\
\nu_{14} & \nu_{23} & - \nu_{32} & - \nu_{41}
\end{pmatrix} =
\begin{pmatrix}
\bar\nu_{12} + \bar\nu_{13}   & - \bar\nu_{21} - \bar\nu_{24} & 
- \bar\nu_{31} - \bar\nu_{34} & \bar\nu_{42} + \bar\nu_{43} \\
\bar\nu_{13} + \bar\nu_{14}   & \bar\nu_{23} + \bar\nu_{24}   &
- \bar\nu_{31} - \bar\nu_{32} & - \bar\nu_{41} - \bar\nu_{42}
\end{pmatrix}
= \Gamma_{\bar\nu} .
\end{align*}

Dynamical equality alone implies $\bar\nu_{21}=\bar\nu_{24}=\bar\nu_{42}=\bar\nu_{43}=0$ (without using complex balance for $\bar\nu$ and positivity of $\nu$).
Hence, instead of $G^\text{com}_{V_s}$, we consider the graph of relevant edges $G^\text{rel}_{V_s} = (V_s,E^\text{rel})$ with $E^\text{rel} = E^\text{com} \setminus \{ 2 \to 1, 2 \to 4, 4 \to 2, 4 \to 3 \}$,
and rewrite dynamical equality in terms of $\nu = v_k(x) \in \R^E_>$ and $\bar\nu = v_{\bar k}(x) \in \R^{E^\text{rel}}_\ge$,
\begin{equation} \tag{DE}
\Gamma_\nu =
\begin{pmatrix}
\nu_{12} & 0        & - \nu_{34} & 0 \\
\nu_{14} & \nu_{23} & - \nu_{32} & - \nu_{41}
\end{pmatrix} =
\begin{pmatrix}
\bar\nu_{12} + \bar\nu_{13}   & 0 & 
- \bar\nu_{31} - \bar\nu_{34} & 0 \\
\bar\nu_{13} + \bar\nu_{14}   & \bar\nu_{23} &
- \bar\nu_{31} - \bar\nu_{32} & - \bar\nu_{41} 
\end{pmatrix}
= \Gamma_{\bar\nu} .
\end{equation}

Further, we write complex balance in terms of $\bar\nu = v_{\bar k}(x) \in \R^{E^\text{rel}}_\ge$,
\begin{equation} \tag{CB}
I_{E^\text{rel}} \, \bar\nu =
\begin{pmatrix}
\bar\nu_{31} + \bar\nu_{41} - \bar\nu_{12} - \bar\nu_{13} - \bar\nu_{14} \\
\bar\nu_{12} + \bar\nu_{32} - \bar\nu_{23} \\
\bar\nu_{13} + \bar\nu_{23} - \bar\nu_{31} - \bar\nu_{32} - \bar\nu_{34} \\
\bar\nu_{14} + \bar\nu_{34} - \bar\nu_{41} 
\end{pmatrix}
= 0 .
\end{equation}

Now, we compute the disguised complex-balanced flux cone,
\begin{equation} \tag{$\mathcal{C}$}
\begin{aligned}
\mathcal{C}^\text{dCB}_{(G,y)} &=
\{ \nu \in \R^E_\ge \mid 
\exists \bar\nu \in \R^{E^\text{rel}}_\ge \colon
I_{E^\text{rel}} \, \bar\nu = 0 
\wedge
\Gamma_\nu= \Gamma_{\bar\nu} 
\} \\
&=
\{ \nu \in \R^E_\ge \mid
\nu_{12} - \nu_{34} =
\nu_{23} + \nu_{14} - \nu_{32} - \nu_{41} = 0 , \\
& \phantom{= \{ \nu \in \R^E_\ge \mid \hspace{5pt}}
| \nu_{41} - \nu_{14} | \le \nu_{12} \le \nu_{41} + \nu_{32}
\} ,
\end{aligned}
\end{equation}
by variable elimination/cone projection,
using the {\tt lrs} (lexicographic reverse search) algorithm~\cite{avis1992pivoting,avis2000lrs,avis2023lrslib}.

Clearly, $\mathcal{C}^\text{dCB}_{(G,y)} \subset \mathcal{C}^\text{flx}_{(G,y)}$ with the (classical) flux cone
\begin{align*}
\mathcal{C}^\text{flx}_{(G,y)} &= 
\{ \nu \in \R^E_\ge \mid 
N \nu = 0 \} \\
&= \{ \nu \in \R^E_\ge \mid
\nu_{12} - \nu_{34} = \nu_{23} + \nu_{14} - \nu_{32} - \nu_{41} = 0 
\} .
\end{align*}

By Theorem~\ref{thm:KdCB_y},
the disguised complex-balanced parameter locus
can be determined as
\begin{equation} \tag{$\mathcal{K}$}
\mathcal{K}^\textnormal{dCB}_{(G,y)} = 
\{ k \in \R^E_> \mid 
\exists \nu \in \mathcal{P} \colon
\nu^z=k^z , \,
\forall 
z \in \mathcal{D}
\} ,
\end{equation}
where
\[
\mathcal{P} = \left( \mathcal{C}^\textnormal{dCB}_{(G,y)} \right)_> \cap \Delta
\quad \text{and} \quad
\mathcal{D} = \ker \binom{Y_s I_{E,s}}{\ones^\trans} .
\]
The ``polytope'' $\mathcal P$ is a normalization of the positive part of the cone $\mathcal{C}^\textnormal{dCB}_{(G,y)}$, obtained by intersecting with the standard simplex $\Delta$, 
and the subspace $\mathcal D$ records affine dependencies between the source complexes (monomials) of the reactions.
For the latter, we specify (support-minimal) basis vectors,
\[
\mathcal{D} = \ker
\bordermatrix{& \gray{12} & \gray{14} & \gray{23} & \gray{32} & \gray{34} & \gray{41} \cr
& 0 & 0 & 1 & 1 & 1 & 0 \cr
& 0 & 0 & 0 & 1 & 1 & 1 \cr
& 1 & 1 & 1 & 1 & 1 & 1
}
= \im
\bordermatrix{& \cr
\gray{12} &-1 & 0 &-1 \cr
\gray{14} & 1 & 0 & 0 \cr
\gray{23} & 0 & 0 & 1 \cr
\gray{32} & 0 &-1 & 0 \cr
\gray{34} & 0 & 1 &-1 \cr
\gray{41} & 0 & 0 & 1
}
.
\]

Altogether, we have to determine $\nu \in \R^E_>$ that fulfill the equality and inequality constraints in $(\mathcal{C})$ and  $(\mathcal{K})$,
\begin{gather*}
\nu_{12} - \nu_{34} = \nu_{23} + \nu_{14} - \nu_{32} - \nu_{41} = 0 , \\
\left( \nu_{12} + \nu_{14} + \nu_{23} + \nu_{32} + \nu_{34} + \nu_{41} = 1 \right) , \\
\frac{\nu_{14}}{\nu_{12}} = \frac{k_{14}}{k_{12}} =: c_1 , \quad 
\frac{\nu_{32}}{\nu_{34}} = \frac{k_{32}}{k_{34}} =: c_3 , \quad 
\frac{\nu_{23} \nu_{41}}{\nu_{12} \nu_{34}} = \frac{k_{23} k_{41}}{k_{12} k_{34}} =: c_\circ ,
\intertext{and}
| \nu_{41} - \nu_{14} | \le \nu_{12} \le \nu_{41} + \nu_{32} .
\end{gather*}

The four homogeneous linear equations
allow us to express all variables in terms of $\nu_{12}$ and $\nu_{41}$,
\[
\nu_{14} = c_1 \nu_{12} , \quad
\nu_{23} = (c_3-c_1) \nu_{12} + \nu_{41} , \quad 
\nu_{32} = c_3 \nu_{34} , \quad
\nu_{34} = \nu_{12} .
\]
The last (homogeneous, nonlinear) equation becomes
\[
\frac{\left( (c_3-c_1) \nu_{12} + \nu_{41} \right) \nu_{41}}{(\nu_{12})^2} = c_\circ ,
\]
and the inequalities take the form
\[
| \nu_{41} - c_1 \nu_{12} | \le \nu_{12} \le \nu_{41} + c_3 \nu_{12} .
\]

All (in-)equalities (except the normalization) are homogeneous, and we obtain a quadratic equation for the ratio $\xi = \nu_{41}/\nu_{12}$,
\[
f(\xi) := \left( c_3-c_1 + \xi \right) \xi = c_\circ ,
\]
constrained by the inequalities
\[
| \xi - c_1 | \le 1 \le \xi + c_3 .
\]
Every admissible solution $\xi$ determines a unique $\nu$ (after normalization) and hence a disguised complex-balanced realization.
Thus, we have reduced the six-dimensional quantifier-elimination problem to the one-dimensional feasibility problem
\[
f(\xi) = c_\circ , 
\quad
\xi \in [\xi^\text{lb},\xi^\text{ub}] ,
\quad \text{where }
\xi^\text{lb} = \max(c_1-1,1-c_3)
\text{ and } \xi^\text{ub} = 1+c_1 .
\]

Crucially, the quadratic function is monotonically increasing on the admissible interval,
\begin{align*}
f'(\xi) &= 2 \xi + c_3 - c_1 \\
&\ge 2 \xi^\text{lb} + c_3 - c_1 \\
&= \max(c_1+c_3-2, 2-c_3-c_1) \\
&= |c_1+c_3-2| \\
&\ge 0 ,
\end{align*}
and solvability is equivalent to
\[
f(\xi^\text{lb}) \le c_\circ \le f(\xi^\text{ub}) .
\]

The boundary values are
\[
f(c_1-1)=f(1-c_3)=(1-c_1)(1-c_3)
\quad \text{and} \quad
f(\xi^\text{ub})=(1+c_1)(1+c_3) ,
\]
and feasibility becomes
\[
(1-c_1)(1-c_3) \le c_\circ \le (1+c_1)(1+c_3) 
\]
or, in the original parameters,
\begin{equation} \tag{$\star$}
\textstyle
(1-\frac{k_{14}}{k_{12}})(1-\frac{k_{32}}{k_{34}}) \le \frac{k_{23} k_{41}}{k_{12} k_{34}} \le (1+\frac{k_{14}}{k_{12}})(1+\frac{k_{32}}{k_{34}}) .
\end{equation}
That is,
\[
\mathcal{K}^\textnormal{dCB}_{(G,y)} = 
\{ k \in \R^E_> \mid (\star) \} .
\]
This characterization was also obtained by Boros et al.~\cite{BorosEtAl2026}. Their derivation starts from Theorem~\ref{thm:KdCB_x} and performs quantifier elimination on the variables $(x,\nu)$ using {\tt Mathematica}, whereas we start from Theorem~\ref{thm:KdCB_y} and perform the quantifier elimination on the variables $\nu$ analytically.

\subsection*{Acknowledgements}

This research was funded in whole, or in part, by the Austrian Science Fund (FWF), 
grant DOI 10.55776/PAT3748324 to SM.

For open access purposes, the authors have applied a CC BY public
copyright license to any author-accepted manuscript version arising
from this submission.

\subsection*{Conflict of interest and data availability}

On behalf of all authors, the corresponding author states 
that there is no conflict of interest
and that the manuscript has no associated data.


\clearpage

\addcontentsline{toc}{section}{References}
\bibliographystyle{abbrv}
\bibliography{old,new,MR}

\end{document}

%% file: fig_run_G.tex
\draw [step=1, gray, very thin] (0,0) grid (1,1);
\draw [ ->, gray, very thin] (0,0)--(1.25,0);
\draw [ ->, gray, very thin] (0,0)--(0,1.25);

\begin{scope}[shift={(0,0)}]

\node[circle, fill=blue, inner sep=1pt, outer sep=5pt] (A) at (0,0) {};
\node[circle, fill=blue, inner sep=1pt, outer sep=5pt] (B) at (1,0) {};
\node[circle, fill=blue, inner sep=1pt, outer sep=5pt] (C) at (1,1) {};
\node[circle, fill=blue, inner sep=1pt, outer sep=5pt] (D) at (0,1) {};


\draw[arrows={-stealth},very thick,blue] (A) to (B);
\draw[arrows={-stealth},very thick,blue,transform canvas={xshift=+2pt}] (B) to (C);
\draw[arrows={-stealth},very thick,blue,transform canvas={xshift=-2pt}] (C) to (B);
\draw[arrows={-stealth},very thick,blue] (C) to (D);
\draw[arrows={-stealth},very thick,blue,transform canvas={xshift=-2pt}] (D) to (A);
\draw[arrows={-stealth},very thick,blue,transform canvas={xshift=+2pt}] (A) to (D);

\node [below left]  (A) at (A) {$1$};
\node [below right] (B) at (B) {$2$};
\node [above right] (C) at (C) {$3$};
\node [above left]  (D) at (D) {$4$};

\node (G) at (0.5,1.5) {$G=(V,E)$};

\end{scope}

%% file: fig_run_Gcom.tex
\draw [step=1, gray, very thin] (0,0) grid (1,1);
\draw [ ->, gray, very thin] (0,0)--(1.25,0);
\draw [ ->, gray, very thin] (0,0)--(0,1.25);

\begin{scope}[shift={(0,0)}]

\node[circle, fill=gray, inner sep=1pt, outer sep=5pt] (A) at (0,0) {};
\node[circle, fill=gray, inner sep=1pt, outer sep=5pt] (B) at (1,0) {};
\node[circle, fill=gray, inner sep=1pt, outer sep=5pt] (C) at (1,1) {};
\node[circle, fill=gray, inner sep=1pt, outer sep=5pt] (D) at (0,1) {};


\draw[arrows={-stealth},very thick,gray,transform canvas={yshift=-2pt}] (A) to (B);
\draw[arrows={-stealth},very thick,gray,transform canvas={yshift=+2pt}] (B) to (A);
\draw[arrows={-stealth},very thick,gray,transform canvas={xshift=+2pt}] (B) to (C);
\draw[arrows={-stealth},very thick,gray,transform canvas={xshift=-2pt}] (C) to (B);
\draw[arrows={-stealth},very thick,gray,transform canvas={yshift=+2pt}] (C) to (D);
\draw[arrows={-stealth},very thick,gray,transform canvas={yshift=-2pt}] (D) to (C);
\draw[arrows={-stealth},very thick,gray,transform canvas={xshift=-2pt}] (D) to (A);
\draw[arrows={-stealth},very thick,gray,transform canvas={xshift=+2pt}] (A) to (D);

\draw[arrows={-stealth},very thick,gray,transform canvas={xshift=+1.4pt,yshift=-1.4pt}] (A) to (C);
\draw[arrows={-stealth},very thick,gray,transform canvas={xshift=-1.4pt,yshift=+1.4pt}] (C) to (A);
\draw[arrows={-stealth},very thick,gray,transform canvas={xshift=+1.4pt,yshift=+1.4pt}] (B) to (D);
\draw[arrows={-stealth},very thick,gray,transform canvas={xshift=-1.4pt,yshift=-1.4pt}] (D) to (B);

\node [below left]  (A) at (A) {$1$};
\node [below right] (B) at (B) {$2$};
\node [above right] (C) at (C) {$3$};
\node [above left]  (D) at (D) {$4$};

\node (G) at (0.5,1.5) {$G^\text{com}_{V_s}$};

\end{scope}

%% file: fig_run_Grel.tex
\draw [step=1, gray, very thin] (0,0) grid (1,1);
\draw [ ->, gray, very thin] (0,0)--(1.25,0);
\draw [ ->, gray, very thin] (0,0)--(0,1.25);

\begin{scope}[shift={(0,0)}]

\node[circle, fill=blue, inner sep=1pt, outer sep=5pt] (A) at (0,0) {};
\node[circle, fill=blue, inner sep=1pt, outer sep=5pt] (B) at (1,0) {};
\node[circle, fill=blue, inner sep=1pt, outer sep=5pt] (C) at (1,1) {};
\node[circle, fill=blue, inner sep=1pt, outer sep=5pt] (D) at (0,1) {};


\draw[arrows={-stealth},very thick,blue] (A) to (B);
\draw[arrows={-stealth},very thick,blue,transform canvas={xshift=+2pt}] (B) to (C);
\draw[arrows={-stealth},very thick,blue,transform canvas={xshift=-2pt}] (C) to (B);
\draw[arrows={-stealth},very thick,blue] (C) to (D);
\draw[arrows={-stealth},very thick,blue,transform canvas={xshift=-2pt}] (D) to (A);
\draw[arrows={-stealth},very thick,blue,transform canvas={xshift=+2pt}] (A) to (D);

\draw[arrows={-stealth},very thick,gray,transform canvas={xshift=+1.4pt,yshift=-1.4pt}] (A) to (C);
\draw[arrows={-stealth},very thick,gray,transform canvas={xshift=-1.4pt,yshift=+1.4pt}] (C) to (A);

\node [below left]  (A) at (A) {$1$};
\node [below right] (B) at (B) {$2$};
\node [above right] (C) at (C) {$3$};
\node [above left]  (D) at (D) {$4$};

\node (G) at (0.5,1.5) {$G^\text{rel}_{V_s}$};

\end{scope}